\documentclass[12pt]{amsart}

\usepackage{amsfonts}
\usepackage{amssymb}
\usepackage{amsthm}
\usepackage{amsmath}
\usepackage{enumerate}
\usepackage{hyperref}
\usepackage{amsrefs}
\usepackage{graphicx}

\newcommand{\IC}{\mathbb{C}}

\newcommand{\set}[1]{\left\{#1\right\}}

\newtheorem{proposition}{Proposition}[section]
\newtheorem{lemma}[proposition]{Lemma}
\newtheorem{theorem}[proposition]{Theorem}
\newtheorem{corollary}[proposition]{Corollary}
\theoremstyle{definition}
\newtheorem{definition}[proposition]{Definition}
\theoremstyle{remark}

\newtheorem{questions}{Questions}
\newtheorem{example}[proposition]{Example}

\newtheorem{remark}[proposition]{Remark}

\numberwithin{equation}{section}
\def\J#1#2#3{ \left\{ #1,#2,#3 \right\} }

\begin{document}

\title{Derivations on ternary rings of operators}
\author{Robert Pluta}
\address{Department of Mathematics, University of California, Irvine, CA 92697-3875, USA}
\email{plutar@tcd.ie}
\author{Bernard Russo}
\address{Department of Mathematics, University of California, Irvine, CA 92697-3875, USA}
\email{brusso@math.uci.edu}

\date{\today}

\begin{abstract}
To each projection $p$ in a $C^*$-algebra $A$
we associate a family of derivations on $A$, called {\em $p$-derivations},
and relate them to the space of triple derivations on $p A (1-p)$.
We then show that every derivation  on a ternary ring of operators is spatial and we investigate
whether every such derivation on a weakly closed ternary ring of operators is inner.

%
\end{abstract}

\maketitle
\section{$S$-derivations on C*-algebras}

If $A$ is a  $C^*$-algebra,
we let $D(A)$ denote the Banach Lie algebra of derivations on $A$.
To be more precise $D(A)$
consists of all operators $\delta \in B(A)$ that satisfy
$\delta(xy) = \delta(x)y + x\delta(x)$
for every $x, y$~in~$A$.  $B(A)$ denotes the bounded linear operators on $A$.

A~derivation~$\delta \in  D(A)$ is called
self-adjoint if $\delta = \delta^*$,
where $\delta^*$ is the derivation  defined by
$\delta^*(x) = \delta(x^*)^*$ for every $x$ in $A$.
The space of all self-adjoint derivations on $A$
is a real Banach Lie subalgebra of $ D(A)$ and is denoted $D^*(A)$.

Derivations on $C^*$-algebras
have suitable counterparts in
a more general setting of ternary rings of operators, or TROs for short,
where they are sometimes termed triple derivations. However, in this paper we shall use the term {\it triple derivation} to denote a derivation of a Jordan triple system. For example, 
if $X$ is a Banach subspace of a $C^*$-algebra
and $x y^*z+zy^*x \in X$ for every $x, y, z$ in $X$,
then $X$ is called a JC$^*$-triple
and a {\em triple derivation} on $X$ is an operator $\tau \in  B(X)$
satisfying
$$\tau(\{xy^*z\}) = \{\tau(x)y^*z\}  + \{ x\tau(y)^*z \} + \{ xy^*\tau(z)\}$$
for every $x, y, z$ in $X$, where $\{xyz\}=(xy^*z+zy^*x)/2$.

We shall use the term TRO-derivation, as follows: If $X$ is a Banach subspace of a $C^*$-algebra
and $x y^*z \in X$ for every $x, y, z$ in $X$,
then $X$ is called a TRO
and a {\em TRO-derivation} on $X$ is an operator $\tau \in  B(X)$
satisfying
$$\tau(xy^*z) = \tau(x)y^*z  +  x\tau(y)^*z  +  xy^*\tau(z)$$
for every $x, y, z$ in $X$.


It is clear that a TRO (resp.\ JC$^*$-triple) can also be defined as a Banach subspace of $B(H,K)$, the bounded operators from Hilbert space $H$ to Hilbert space $K$, which is closed under the triple product $xy^*z$ (resp.\ $(xy^*z+zy^*x)/2$).  If a TRO is weakly closed, it is called a W$^*$-TRO.

In this section
we will introduce the class of $S$-derivations on a $C^*$-algebra $A$
associated with a subspace  $S\subseteq A$.
Of particular interest will be the case $S = pAp$ for a projection $p$ in $A$. 
We  will seek to determine 
the relationship between 
the class of $pAp$ derivations (which we call $p$-derivations for short) on $A$ 
and the class of TRO-derivations on $pA(1-p)$.

\begin{definition}
Let $A$ be a $C^*$-algebra and let $S$  be a  subspace of~$A$.
We~say~that a derivation $\delta \in  D(A)$
is  {\em associated  with} $S$, or simply that $\delta$ is an {\em $S$-derivation},
if $\delta$  leaves $S$   invariant   in the sense that $\delta(S) \subseteq S$.
\end{definition}

We use $ D_S(A)$ to denote the set of all $S$-derivations.
In order to simplify the notation,
we write $ D_e(A)$ for $ D_{eAe}(A)$
in case $S = eAe$, for some idempotent $e\in A$,  
and we abuse the terminology slightly
by referring to the elements of $ D_e(A)$
simply as {\em $e$-derivations}.

\medskip

To repeat, 
given an arbitrary idempotent $e$ in a $C^*$-algebra $A$,  
which in particular may be a projection,  
by an $e$-derivation on $A$ we mean a derivation $\delta \in  D(A)$ satisfying 
$\delta (eAe) \subseteq eAe$. 
This condition is easily seen to be equivalent to
the requirement that $\delta(e) = 0$.

\medskip


\begin{example}
Let $A$ be a $C^*$-algebra and let $e\in A$ be~an~idempotent.
Fix $a \in eAe$ and $b \in (1-e)A(1-e) = \set{x - xe - ex + exe \,:\, x \in A}$.
Then $\delta \colon A \to A$  defined by
$\delta(x) = (a + b)x - x(a + b)$
is an $e$-derivation.
\end{example}

\begin{lemma}
Let $A$ be a $C^*$-algebra and let $S$ be a subalgebra 
with an identity element $1_S$
(possibly different from the identity~element~of~$A$ if $A$ is unital).
Let $\delta \in  D(A)$ be a derivation.
The following statements~hold.
\begin{enumerate}
\item If  $\delta(S) \subseteq S $ then $\delta(1_S) = 0$.
\item If  $\delta(1_S) = 0 $ then $\delta(S) \subseteq 1_S A 1_S$.
\end{enumerate}
\end{lemma}

\begin{proof}
A straightforward consequence of the derivation property.
\end{proof}

\begin{lemma}
Let $A$ be a $C^*$-algebra and let $e\in A$ be an idempotent.
Let $\delta \in  D(A)$ be a derivation.
The following statements hold.
\begin{enumerate}
\item[$(1)$] If  $\delta(e) = 0$,
then $\delta$ leaves invariant the following subspaces
$$ eAe, \quad   eA(1-e), \quad   (1-e)Ae,   \quad  (1-e)A(1-e).$$
\item[$(2)$] If $\delta$  leaves invariant  $eAe$ or $(1-e)A(1-e)$, then $\delta(e) = 0$. 
\end{enumerate}
Additionally, let $\delta = \delta^*$ and $e = e^*$. 
Then the following statement~holds. 
\begin{enumerate}
\item[$(3)$] If $\delta$  leaves invariant  $eA(1-e)$  or $(1-e)Ae$, then $\delta(e) = 0$. 
\end{enumerate}
\end{lemma}

\begin{proof}
The assertions (1) and (2) are straightforward 
consequences of the derivation property. 
To prove (3),  
assume that $eA(1-e)$ is invariant for $\delta = \delta^*$, and $e = e^*$.
Since $\delta(e) ~ = ~ \delta(e)e + e \delta(e)$, we have $e \delta(e)e = 0$ 
and hence 
\begin{equation}
\delta(e) ~ = ~ e \delta(e) (1 - e) + (1 - e) \delta(e) e.  
\nonumber  
\end{equation}   
This shows that both $e\delta(e)$ and $\delta(e) (1-e)$  are  equal to $e \delta(e) (1 - e)$, 
and so both $e\delta(e)$ and $\delta(e) (1-e)$ are elements of the subspace $eA(1-e)$  
which is invariant  under  $\delta$.

We will show that $\delta(e) = 0$ by showing that $\delta(e)^2 = 0$.  
For this, we identify $A$ with 
$  
\left(\begin{smallmatrix}
eAe     &  eA(1-e)  \\
(1-e)Ae &  (1-e)A(1-e) 
\end{smallmatrix}\right)
$
and write $\delta(e)$ and $\delta^2(e)$ as
$$  
\delta(e) = 
\left(\begin{smallmatrix}
0 & e \delta(e) (1 - e)  \\
(1 - e) \delta(e) e & 0
\end{smallmatrix}\right), 
\qquad \quad 
\delta^2(e) = 
\left(\begin{smallmatrix}
a & b  \\
c & d 
\end{smallmatrix}\right). 
$$
Then    
$
\delta(e)^2 
= 
\left(
\begin{smallmatrix}
e\delta(e)(1 - e) \delta(e) e  &  0  \\
0 & (1-e)\delta(e) e \delta(e) (1 - e) 
\end{smallmatrix}
\right)
$ 
and since 
\begin{eqnarray*}
\delta( e \delta(e) ) & = & 
\left(\begin{smallmatrix}
e\delta(e)(1 - e) \delta(e) e +  ea  &  eb  \\
0 & (1-e)\delta(e) e \delta(e) (1 - e) 
\end{smallmatrix}\right) 
~~ \in  ~~  
\left(\begin{smallmatrix}
0 & eA(1-e)  \\
0 & 0 
\end{smallmatrix}\right),   \\ 
\delta(  \delta(e) (1-e)  ) & = &   
\left(\begin{smallmatrix}
- e\delta(e)(1 - e) \delta(e) e  & b(1-e) \\
0 & d(1-e) - (1-e)\delta(e) e \delta(e) (1 - e) 
\end{smallmatrix}\right) 
~~ \in  ~~  
\left(\begin{smallmatrix}
0 & eA(1-e)  \\
0 & 0 
\end{smallmatrix}\right),  
\end{eqnarray*}
it follows that 
$(1-e)\delta(e) e \delta(e) (1 - e)  = 0 =  e\delta(e)(1 - e) \delta(e)e$. 
Thus $\delta(e)^2 = 0$, as desired. 
\end{proof}

If $A$ is a $C^*$-algebra and $p\in A$ is a projection, 
we let $ D^*_p(A)$
denote the (real) Banach Lie algebra of self-adjoint $p$-derivations on $A$. 
To be more precise $ D^*_p(A)$ consists of all  derivations $\delta \in   D(A)$ 
that satisfy   $\delta(p) = 0$ and $\delta = \delta^*$. 
If $X$ is a TRO, we use  $ D_{TRO}(X)$
to denote the (real) Banach Lie algebra of all TRO-derivations on $X$.

\begin{remark}\label{rem:0125161}
Let $A$ be a unital $C^*$-algebra and let $p\in A$ be a projection. 
Then the map $$  \Delta \colon   D^*_p(A)   \to   D_{TRO}(pA(1-p)), \qquad \Delta(\delta) = \delta_{|pA(1-p)}$$
is a homomorphism of Banach Lie algebras. 
\end{remark}

\begin{example}
\noindent Let $A = M_2(\IC)$,  
$p = 
\left(\begin{smallmatrix}
1  & 0 \\
0 & 0 
\end{smallmatrix}\right). 
$ 
The set of all $p$-derivations on $A$ is:
\begin{eqnarray*}
 D_p(A) & = &  \set{ \delta \in  D(A)  \,:\,    \delta(p) = 0  } \\ 
& \simeq &
\set{  
\left(
\begin{smallmatrix}
\alpha   &  0 \\
0 & \beta   
\end{smallmatrix}
\right) 
\,:\, \alpha, \beta  \in \IC}  \\
 & = &  \text{a complex  Banach Lie algebra.} 
\end{eqnarray*}

\bigskip

\noindent  The set of all self-adjoint $p$-derivations  is:  
\begin{eqnarray*}
 D^*_p(A) & = &  \set{ \delta \in  D_p(A)  \,:\,    \delta  =  \delta^*  } \\ 
& \simeq &
\set{  
\left(
\begin{smallmatrix}
\alpha   &  0 \\
0 & \beta   
\end{smallmatrix}
\right) 
\,:\, \alpha, \beta  \in \IC \text{ with }  \Re(\alpha)  =  \Re(\beta)        }   \\
 & = &  \text{a real Banach Lie algebra.}  
\end{eqnarray*} 

\bigskip

\noindent 
\noindent The mapping 
$$ \Delta \colon  D^*_p(A) \to  D_{TRO}(X), \qquad \Delta(\delta) = \delta_{|X} $$
defines a linear surjection between the self-adjoint $p$-derivations on $A$ 
and the TRO-derivations on  
$X = pA(1-p) =  
\left(
\begin{smallmatrix}
0  &  \IC \\
0 & 0   
\end{smallmatrix}
\right)$ (see Lemma~\ref{lem:0125161}).  
The   kernel of $\Delta$ is isomorphic to the center of $A$, i.e., 
$$\ker \Delta = Z(A)  =   
\set{  
\left(
\begin{smallmatrix}
\alpha   &  0 \\
0 & \alpha    
\end{smallmatrix}
\right) 
\,:\, \alpha  \in \IC}. $$
In other words, the TRO-derivations on 
$X = pA(1-p) =  
\left(
\begin{smallmatrix}
0  &  \IC \\
0 & 0   
\end{smallmatrix}
\right) $ 
are precisely the   
self-adjoint $p$-derivations on 
the linking algebra 
$
\left(
\begin{smallmatrix}
XX^*  &  X \\
X^* & X^*X   
\end{smallmatrix}
\right)
= A = M_2(\IC)$.  
\end{example}

\begin{example}
\noindent Let $A = M_5(\IC)$,
and let $p\in A$ be
the projection matrix
with 1 in the $(1, 1)$ and $(2, 2)$ position and zero's elsewhere.
The set of all $p$-derivations on $A$ is:
\begin{eqnarray*}
 D_p(A) & = &  \set{ \delta \in  D(A)  \,:\,    \delta(p) = 0  } \\
& \simeq &
\set{
\left(
\begin{smallmatrix}
A &  0 \\
0 & B
\end{smallmatrix}
\right)
\,:\, A\in M_2(\IC), B\in M_3(\IC)}  \\
 & = &  \text{a complex  Banach Lie algebra.}
\end{eqnarray*}
The set of all self-adjoint $p$-derivations  is
$ D^*_p(A)  =   \set{ \delta \in  D_p(A)  \,:\,    \delta  =  \delta^*  } $
and it can be identified with the real Banach Lie algebra
consisting of all matrices of the form
$\left(
\begin{smallmatrix}
A  &  0 \\
0 & B
\end{smallmatrix}
\right)
$  where
$A \in M_2(\IC)$, $B\in M_3(\IC)$,
and
$
\left(
\begin{smallmatrix}
A + A^*  &  0 \\
0 & B + B^*
\end{smallmatrix}
\right)
$
is in the  center of $A$.
\end{example}

\section{Derivations on TROs}

If $A$ is a unital C*-algebra and $e$ is a projection in $A$, then $X:=eA(1-e)$ is a TRO.  Conversely if $X\subset B(K,H)$ is a TRO, then with $X^*=\{x^*:x\in X\}\subset B(H,K)$, $XX^*=\hbox{span}\, \{ xy^*:x,y\in X\}\subset B(H)$, $X^*X=\hbox{span}\, \{z^*w: z,w\in X\}\subset B(K)$, $K_l(X)=\overline{XX^*}^n $, $K_r(X)=\overline{X^*X}^n$, we let\footnote{If $K_l(X)$ and $K_r(X)$ are unital subalgebras of $B(H)$ and $B(K)$ (resp.), and $X$ is nondegenerate, that is, $XX^*X$ is dense in $X$, then we take $A_X$ to be $\left[\begin{smallmatrix}
K_l(X)&X\\
X^*&K_r(X)
\end{smallmatrix}\right]$}
\[
A_X=\left[\begin{matrix}
K_l(X)+\IC 1_H&X\\
X^*&K_r(X)+\IC 1_K
\end{matrix}\right]\subset B(H\oplus K)
\]
denote the (unital) linking C*-algebra of $X$.  Then we have a TRO-isomorphism $X\simeq eA_X(1-e)$, where $e=\left[\begin{smallmatrix} 1&0\\ 0&0\end{smallmatrix}\right]$.

\begin{lemma}\label{lem:0125161}
Let $X$ be a TRO and let $D:X\rightarrow X$ be a TRO-derivation of $X$.  If $A_0=\left(\begin{smallmatrix}
XX^* & X  \\
X^* & X^*X
\end{smallmatrix}\right)$, then the map $\delta_0:A_0\rightarrow A_0$ given by
\[
\left(\begin{matrix}
\sum_i  x_iy_i^* & x  \\
y^* & \sum_j z_j^*w_j
\end{matrix}\right)\mapsto \left(\begin{matrix}
\sum_i(x_i(Dy_i)^*+(Dx_i)y_i^*) & Dx  \\
(Dy)^* & \sum_j(z_j^*(Dw_j)+(Dz_j)^*w_j)
\end{matrix}\right)
\]
is well defined and a bounded  *-derivation of $A_0$, which extends $D$ (when $X$ is embedded in $A_X$ via $x\mapsto  \left(\begin{smallmatrix}
0& x  \\
0& 0
\end{smallmatrix}\right)$), and which itself extends to a *-derivation $\delta$ of $A_X$.
Thus,  the Lie algebra homomorphism $\Delta: \delta\mapsto \delta_{|X}$ given in Remark~\ref{rem:0125161} is onto.

\end{lemma}
\begin{proof}
If $\sum_ix_iy_i^*=0$, then for every $z\in X$,  
\begin{eqnarray*}0&=& D(\sum_i  x_iy_i^*z)\\
&=&\sum_i((Dx_i)y_i^*z+x_i(Dy_i)^*z+x_iy_i^*(Dz))\\
&=&(\sum_i((Dx_i)y_i^*+x_i(Dy_i)^*))z.
\end{eqnarray*} Since this is true for every $z$, we have $\sum_i((Dx_i)y_i^*+x_i(Dy_i)^*)=0$ (see \cite[Lemma 2.3(iv)]{Hamana99}) and
it follows that $\delta_0$ is well defined.\medskip

The map $\delta_0$ is self-adjoint since if $a=\left(\begin{smallmatrix}
\sum_i  x_iy_i^* & x  \\
y^* & \sum_j z_j^*w_j
\end{smallmatrix}\right)$, then 
\begin{eqnarray*}
\delta_0(a^*)&=&
\delta_0\left(\begin{matrix}
\sum_i  y_ix_i^* & y  \\
x^* & \sum_j w_j^*z_j
\end{matrix}\right)\\
&=&\left(\begin{matrix}
\sum_i(y_i(Dx_i)^*+(Dy_i)x_i^*) & Dy  \\
(Dx)^* & \sum_j(w_j^*(Dz_j)+(Dw_j)^*z_j)
\end{matrix}\right)\\
&=&\delta_0(a)^*.
\end{eqnarray*}

It is easy to verify that $\delta_0(a^2)=\delta_0(a)a+a\delta_0(a)$ so that $\delta_0$ is a Jordan *-derivation of $A_0$. (We omit that calculation.)

To see that $\delta_0$ is bounded, we first note that $D$ is bounded, since it is a Jordan triple derivation on the JB*-triple $X$ with the Jordan triple product $\{xyz\}=(xy^*z+zy^*x)/2$, and hence bounded by the theorem of Barton and Friedman \cite{BarFri90}.    Now denoting $\sum_i(x_i(Dy_i)^*+(Dx_i)y_i^*)$ by $\alpha$, we have (by \cite[Lemma 2.3(iv)]{Hamana99} again) $\|\sum_i(x_i(Dy_i)^*+(Dx_i)y_i^*)\|=$
\begin{eqnarray*}
\| \alpha\|
&=&\sup_{\|z\|\le 1,z\in X}\|\alpha z\|\\
&=&\sup_{\|z\|\le 1,z\in X}\|\alpha z+\overbrace{\sum_ix_iy_i^*(Dz)-\sum_ix_iy_i^*(Dz)}^{=0}\|\\
&=&\sup_{\|z\|\le 1,z\in X}\|\sum_i D(x_iy_i^*z)-\sum_ix_iy_i^*(Dz)\|\\
&=&\sup_{\|z\|\le 1,z\in X}\|D\sum_i x_iy_i^*z-\|D\|\sum_ix_iy_i^*\frac{Dz}{\|D\|}\|\\
&\le&2\|D\|\|\sum_i x_iy_i^*\|.
\end{eqnarray*}
Thus $\delta_0$ is bounded and therefore extends to a bounded Jordan *-derivation $\delta$ of $\overline{A_0}^n$ and hence to $A_X$ by setting $\delta(e)=0$, where $e=\left[\begin{smallmatrix} 1_H&0\\ 0&0\end{smallmatrix}\right]$.  By the theorem of Sinclair (\cite[Theorem 3.3]{Sinclair70}), $\delta$ is a derivation of $A_X$.\footnote{It is also easy to verify directly, by (a more involved) calculation, that $\delta_0$ is a derivation, thereby avoiding the use of Sinclair's theorem}
\end{proof}

For any C*-algebra $A\subset B(H)$, the Lie algebra homomorphism $\overline{A}^w\ni z\mapsto \hbox{ad}\, z\in  D(\overline{A}^w)$ is onto (theorem of 
Kadison and Sakai (\cite[4.1.6]{sakai})) and so we have the Lie algebra isomorphism
\[
\overline{A}^w/Z(\overline{A}^w)\simeq   D(\overline{A}^w).
\]
It follows (cf.\ \cite[4.1.7]{sakai}) that
\[
\{t\in  \overline{A}^w:\hbox{ad}\, t(A)\subset A\}/Z(\overline{A}^w)\simeq  D(A),
\]
and
\[
\{t\in  \overline{A}^w:t^*=-t, \hbox{ad}\, t(A)\subset A\}/Z(\overline{A}^w)\simeq  D^*(A).
\]
Further, for a projection $e$ in $A$, we have
\[
\{t\in  \overline{A}^w:et=te,t^*=-t, \hbox{ad}\, t(A)\subset A\}/Z(\overline{A}^w)\simeq  D_e^*(A).
\]

Using these facts in the setting of Lemma~\ref{lem:0125161}, and noting that, by  
\cite[page 268]{KaurRuan02}, $\overline{A_X}^w=A_X^{\prime\prime}=\left[\begin{smallmatrix}K_l(X)^{\prime\prime}&\overline{X}^w\\ \overline{X^*}^w&K_r(X)^{\prime\prime}\end{smallmatrix}\right]$, we can now prove the following theorem.

\begin{theorem}\label{thm:0128161}
Every TRO-derivation of a TRO $X$ is spatial in the sense that there exist $\alpha\in K_l(X)^{\prime\prime}$ and $\beta\in K_r(X)^{\prime\prime}$ such that $\alpha^*=-\alpha,\ \beta^*=-\beta$, and $Dx=\alpha x+x\beta$ for every $x\in X$.
\end{theorem}
\begin{proof}
If $D\in  D_{TRO}(X)$, choose $\delta=\hbox{ad}\, t$ for some $t\in\overline{A_X}^w$ with $t^*=-t$, $te=et$ and \[
\left[\begin{matrix} 0&Dx\\ 0&0\end{matrix}\right]=\delta\left[\begin{matrix}0&x\\ 0&0\end{matrix}\right].
\]
The conditions on $t$ imply that $t=\left[\begin{smallmatrix}\alpha&0\\ 0&\beta\end{smallmatrix}\right]$ with $\alpha^*=-\alpha$ and $\beta^*=-\beta$. Moreover
\[
\delta\left[\begin{matrix}0&x\\ 0&0\end{matrix}\right]=\left[\begin{matrix}\alpha &0\\ 0&\beta\end{matrix}\right]\left[\begin{matrix}0&x\\ 0&0\end{matrix}\right]-\left[\begin{matrix}0&x\\ 0&0\end{matrix}\right]\left[\begin{matrix}\alpha&0\\ 0&\beta\end{matrix}\right]=\left[\begin{matrix}0&\alpha x+x(-\beta)\\ 0&0\end{matrix}\right].
\]
\nopagebreak
\end{proof}

A TRO derivation $D$ of a TRO $X$ is said to be an {\it inner TRO derivation} if there exist $\alpha=-\alpha^*\in XX^*$ and $\beta=-\beta^*\in X^*X$ such that $Dx=\alpha x+x\beta$ for $x\in X$.  Note that  there exist $a_i,b_i,c_j,d_j\in X$, $1\le i\le n,1\le j\le m$ such that
$\alpha=\sum_{i=1}^n(a_ib_i^*-b_ia_i^*)$ and $\beta=\sum_{j=1}^m(c_j^*d_j-d_j^*c_j)$.

\begin{corollary}\label{cor:0307161}
Every TRO derivation of a C$^*$-algebra $A$ is of the form $A\ni x\mapsto \alpha x+x\beta$ with elements $\alpha,\beta\in \overline{A}^w$ with $\alpha^*=-\alpha, \beta^*=-\beta$.  In particular, every TRO derivation of a von Neumann algebra is an inner TRO derivation\end{corollary}

Thus, every
W$^*$-TRO  which is TRO-isomorphic to a von Neumann algebra has only inner TRO derivations. For example, this is the case for the stable W$^*$-TROs of \cite{Ruan04} (see subsection~\ref{subsec:0426161}) and the weak*-closed right ideals in certain continuous von Neumann algebras acting on separable Hilbert spaces (see Theorem~\ref{prop:0307162}).
\smallskip

Theorem~\ref{thm:0128161} is an improvement of  \cite{Zalar95}, in which, although proved for the slightly more general case of derivation pairs,  it is assumed that the TRO (called B*-triple system in \cite{Zalar95}) contains the finite rank operators.  For the extension of Zalar's result to unbounded operators, see \cite{Timmermann2000}. \smallskip

A triple derivation $\delta$ of a JC$^*$-triple  $X$ is said to be an {\it inner triple derivation} if there exist finitely many elements $a_i,b_i\in X$, $1\le i\le n$,  such that $\delta x=\sum_{i=1}^n(\{a_ib_ix\}-\{b_ia_ix\})$ for $x\in X$, where $\{xyz\}=(xy^*z+zy^*x)/2$.  For convenience, we denote the inner triple derivation $x\mapsto \{abx\}-\{bax\}$ by $\delta(a,b)$. Thus $$\delta(a,b)(x)=(ab^*x+xb^*a-ba^*x-xa^*b)/2.$$

Let $X$ be a TRO.  As noted in the proof of Lemma~\ref{lem:0125161}, $X$ is a JC$^*$-triple in the triple product $(xy^*z+zy^*x)/2$, and every TRO-derivation of $X$ is obviously a triple derivation.  On the other hand, every inner triple derivation is an inner TRO-derivation.  Indeed, if
$\delta(x)=\{abx\}-\{bax\}$, for some $a,b\in X$, then
$\delta(x)=Ax+xB$, where $A=ab^*-ba^*\in XX^*$, $B=b^*a-a^*b\in X^*X$ with $A,B$ skew-hermitian. Moreover, since by \cite[Theorem 4.6]{BarFri90}, every triple derivation $\delta$ on $X$ is the strong operator limit of a net $\delta_\alpha$ of inner triple derivations, hence TRO-derivations, we have (i) and (ii) in the following proposition.

\begin{proposition}\label{prop:0307161}
Let $X$ be a TRO. 
\begin{description}
\item[(i)] Every TRO-derivation is the strong operator limit of inner TRO-derivations.
\item[(ii)]  The triple derivations on $X$ coincide with the TRO-derivations. 
\item[(iii)] The inner triple derivations on $X$ coincide with the inner TRO-derivations
\item[(iv)] All TRO derivations of $X$ are inner, if and only if, all triple derivations of $X$ are inner.
\end{description}
\end{proposition}
\begin{proof}
Since (iv) is immediate from (ii) and (iii), we only need to show part of (iii), that is,  that every inner TRO-derivation is an inner triple derivation.  If $D$ is an inner TRO-derivation, then $Dx=\alpha x+x\beta$, with $\alpha^*=-\alpha\in XX^*$ and $\beta^*=-\beta\in X^*X$.  We must show that there exist elements $a_k, b_k$ such that $Dx=\sum_{k=1}^p\delta(a_k,b_k) x$ where $\delta(a_k,b_k)$ is the inner triple derivation $x\mapsto \{a_kb_kx\}-\{b_ka_kx\}$.  If $\alpha=\sum_{i=1}^nx_iy_i^*$ and $\beta=\sum_{j=1}^mz_j^*w_j$, then it suffices to take $p=m+n$ and choose $a_i=x_i/2,\ b_i=y_i$ for $1\le i\le n$ and $a_{n+i}=w_i,\ b_{n+i}=z_i/2$ for $1\le i\le m$.
\end{proof}

\section{Derivations on W*-TROs}

A von Neumann algebra $M$  is an example of a unital reversible $JW^*$-algebra, and as such, by \cite[Theorem 2 and the first sentence in its proof]{HoMarPeRu}, every triple derivation on $M$ is an inner triple derivation.  Hence we see that the last statement in Corollary~\ref{cor:0307161} follows also from this and  Proposition~\ref{prop:0307161}(iv).  For completeness, we include a proof of the former result which avoids much of the Jordan theory, starting with the following lemma, the first part of which is straightforward.

 \begin{lemma}\label{lem:0412171}
 Let $A$ be a unital Banach $^*$-algebra equipped with the ternary product given by $\J abc =
\frac12 \ ( a b^* c + cb^* a) $ and the Jordan product $a\circ b=(ab+ba)/2$.

\begin{itemize} 
\item Let $D$ be an inner derivation, that is, $D =\hbox{ ad}\ a:x\mapsto ax-xa$, for some
$a$ in $A$. 
Then
$D=\hbox{ad}\ a$ is a *-derivation whenever $a^*=-a$. Conversely, if
$D$ is a *-derivation, then $a^*=-a+z$ for some $z$ in the center of $A$.
\item Every triple derivation is the sum of a Jordan *-derivation and an inner triple derivation.
\end{itemize}
\end{lemma}
\begin{proof}
To prove the second statement, we modify the proof in \cite[Section 3]{HoPerRus13} which is in a different context. We note first that for a triple derivation $\delta$, $\delta(1)^*=-\delta(1)$. Next, for a triple derivation $\delta$, the mapping $\delta_1(x)=\delta(1)\circ x$ is equal to the inner
triple derivation $-\frac{1}{4}\delta(\delta(1),1)$ so that $\delta_0:=\delta-\delta_1$ is a triple derivation with $\delta_0(1)=0$.  Finally, any triple derivation which vanishes at 1 is a Jordan *-derivation.
\end{proof}
\begin{theorem}\label{thm:0228161}
Every triple derivation on a von Neumann algebra is an inner triple derivation.
\end{theorem}
\begin{proof} 
 It suffices, by the second statement in Lemma~\ref{lem:0412171}, to show that every self-adjoint Jordan derivation is an inner triple derivation.  If $\delta$ is a self-adjoint Jordan derivation of $M$, then  $\delta$ is an associative derivation (by the theorem of Sinclair, \cite[Theorem 3.3]{Sinclair70})) and hence by the theorem of 
Kadison and Sakai (\cite[4.1.6]{sakai}) and  the first statement in Lemma~\ref{lem:0412171}, $\delta (x)=ax-xa $ where $a^*+a=z$ is a self adjoint element of the center of $M$.
Since for every von Neumann algebra, we have  $M=Z(M)+[M,M]$, where $Z(M)$ denotes the center of $M$ (see \cite[Section 3]{PluRus15} for a discussion of this fact), we can  therefore write
\begin{eqnarray*}
a&=&z'+\sum_j[b_j+ic_j,b_j^\prime+ic_j^\prime]\\
&=&z'+\sum_j([b_j,b_j^\prime]-[c_j,c_j^\prime])+i\sum_j([c_j,b_j^\prime]+[b_j,c_j^\prime]),
\end{eqnarray*}
where $b_j,b_j^\prime,c_j,c_j^\prime$ are self adjoint elements of $M$ and  $z'\in Z(M)$.  

It follows that
\[
0=a^*+a-z=(z^\prime)^*+z^\prime-z+2i\sum_j([c_j,b_j^\prime]+[b_j,c_j^\prime])
\]
so that $\sum_j([c_j,b_j^\prime]+[b_j,c_j^\prime])$ belongs to the center of $M$.  We now have
\[\delta=\hbox{ad}\, a=\hbox{ad}\, \sum_j([b_j,b_j^\prime]-[c_j,c_j^\prime]).
\]
A direct calculation shows that 
$\delta$ is equal to the inner triple derivation $\sum_j\left(\delta(b_j,2b_j^\prime)-\delta(c_j,2c_j^\prime)\right)$, completing the proof.\end{proof}

\subsection{Weakly closed right ideals in von Neumann algebras}\label{subs:0307161}
In this subsection, we shall consider the TRO $pM$ where $M$ is a von Neumann algebra and $p$ is a projection in $M$.\smallskip

A TRO of the form $pM$, with $M$ a continuous von Neumann algebra,  is classified into four types in \cite{Horn88} as follows.

\begin{itemize}
\item $II_1^a$ if $M$ is of type $II_1$ and $p$ is (necessarily) finite.
\item $II_{\infty,1}^a$ if $M$ is of type $II_\infty$ and $p$ is a finite projection.
\item $II^a_\infty$ if $M$ is of type $II_\infty$ and $p$ is a properly infinite projection.
\item $III^a$ if $M$ is of type III and $p$ is a (necessarily) properly infinite projection.
\smallskip

Similarly, we also define types for $pM$ for $M$ of type I:
\smallskip

\item $I^a_1$ if $M$ is finite of type $I$ and $p$ is (necessarily) finite.
\item $I^a_{\infty,1}$ if $M$ is of type $I_\infty$ and $p$ is a finite projection.
\item $I^a_\infty$ if $M$ is of type $I_\infty$ and $p$ is a properly  infinite projection.
\end{itemize}

The following theorem involves the cases $II_{\infty}^a,III^a$ and when $M$ is a factor, the cases  $I^a_1,I^a_{1,\infty}$,  and $I^a_\infty$.

\begin{theorem}\label{prop:0307162}
Let $X=pM$ be a TRO, where $M$ is a von Neumann algebra and $p$ is a projection in $M$. 
\begin{description}
\item[(i)] If $X$ is of type $II_{\infty}^a$ or $III^a$, and has a separable predual, then every TRO-derivation of $X$ is an inner TRO-derivation. 
\item[(ii)] If $M$ is of type III and countably decomposable, then every TRO-derivation of $X=pM$ is an inner TRO-derivation. 
\item[(iii)] If $M=B(H)$ is a factor of type I, then
\begin{enumerate}
\item If $\dim H<\infty$, then every TRO-derivation of $X=pM$ is an inner TRO-derivation.
\item If $\dim pH=\dim H$, then every TRO-derivation of $X=pM$ is an inner TRO-derivation.
\item If $\dim pH<\dim H=\infty$, then $X=pM$ admits outer  TRO-derivations.
\end{enumerate}
\end{description}
\end{theorem}
\begin{proof}
If $M$ is a continuous von Neumann algebra with a separable predual  and $p$ is a properly infinite projection in $M$, then it is shown in \cite[Theorem 5.16]{Horn88} that $pM$ is triple isomorphic to a von Neumann algebra, and hence by Theorem~\ref{thm:0228161}, every triple derivation is an inner triple derivation in this case. Consequently, by Proposition~\ref{prop:0307161}(iv), every TRO-derivation is an inner TRO-derivation. (Another way to see this latter fact is to note that by \cite[Lemma 5.15]{Horn88}, $pM$ is actually TRO-isomorphic to a von Neumann algebra, and to apply Corollary~\ref{cor:0307161}.) This proves (i).
  \smallskip

To prove (ii), we note first that
if $A$ is a von Neumann algebra with a projection $p\sim 1$, then $pA$ is TRO-isomorphic to $A$.
Indeed, If $u$ is a partial isometry in $A$ with $uu^*=p$ and $u^*u=1$, then $x\mapsto u^*x$ is a TRO-isomorphism from $pA$ onto $A$. Now if $A$ is of type III, then $\tilde A:=c(p)A$ is of type III, $c(p)$ is the identity of $\tilde A$ and
$pA=p\tilde A$.  Further, if $A$ is countably decomposable, then by \cite[2.2.14]{sakai}, since in $\tilde A$, $c(p)=1_{\tilde A}=c(1_{\tilde A})$, we have $p\sim 1_{\tilde A}$, so $\tilde A$ is TRO-isomorphic to $p\tilde A=pA$.\smallskip

Finally,  let $M=B(H)$. The first statement in (iii) follows from the fact that every finite dimensional semisimple Jordan triple system has only inner derivations.  This result first appeared in \cite[Chapter 11]{Meyberg72} (for an outline of a proof, see  \cite[Theorem 2.8,p.\ 136]{RussoAlmeria} and for the definitions of Jordan triple system and semisimple, see \cite[Section 1.2]{Chubook}).
If  $\dim pH=\dim H$, then $pM\simeq B(H)$ has only inner triple derivations by Theorem~\ref{thm:0228161}.  On the other hand, if 
$\dim pH<\dim H=\infty$, then $pM\simeq B(H,pH)$ has outer triple derivations, as shown in \cite[Corollary 3]{HoMarPeRu}.  By Proposition~\ref{prop:0307161}(iv), this proves (iii)
\end{proof}

\begin{remark}\label{rem:3.5}
Although it follows from Theorem~\ref{prop:0307162}, it is worth pointing out that the TROs $B(\IC,H)$ and $B(H,\IC)$ support outer TRO derivations if and only if $\dim H=\infty$.
According to \cite[Lemma 5.15]{Horn88}, if $B$ is a von Neumann algebra of type $II_\infty$ or III, and $H$ is a separable Hilbert space, then $B$ and $B\overline\otimes B(\IC,H)$ are TRO-isomorphic.  Corollary~\ref{cor:0307161} shows that $B\overline\otimes B(\IC,H)$ has only inner TRO-derivations and only inner triple derivations, although, as just noted,  $B(\IC,H)$ can have an outer TRO derivation and an outer triple derivation. This contrasts the situation of derivations on tensor products of C$^*$-algebras, as in  \cite[Proposition 3.2]{Batty78}).
\end{remark}

\subsection{W*-TROs of types I,II,III}\label{subsec:0426161}

We begin by recalling some concepts from \cite{Ruan04}.
If $R$ is a von Neumann algebra and $e$ is a projection in $R$, then $V:=eR(1-e)$ is a W*-TRO.  Conversely if $V\subset B(K,H)$ is a W*-TRO, then with $V^*=\{x^*:x\in V\}\subset B(H,K)$, $M(V)=\overline{XX^*}^{sot} \subset B(H)$, $N(V)=\overline{X^*X}^{sot} \subset B(K)$, we let 
\[
R_V=\left[\begin{matrix}
M(V)&V\\
V^*&N(V)
\end{matrix}\right]\subset B(H\oplus K)
\]
denote the  linking von Neumann algebra of $V$.  Then we have a SOT-continuous TRO-isomorphism $V\simeq eRe^\perp$, where $e=\left[\begin{smallmatrix} 1_H&0\\ 0&0\end{smallmatrix}\right]$ and $e^\perp=\left[\begin{smallmatrix} 0&0\\ 0&1_K\end{smallmatrix}\right]$.\smallskip

A W*-TRO $V$ is {\it stable} if it is TRO-isomorphic to $B(\ell_2)\overline\otimes V$.   A W*-TRO is of type I,II,or III, by definition, if its linking von Neumann algebra is of that type as a von Neumann algebra.  There is a further classification of the types I and II depending on the types of $M(V)$ and $N(V)$ leading to the types $I_{m,n}, II_{\alpha,\beta}$ where $m,n$ are cardinal numbers and $\alpha,\beta\in \{1,\infty\}$.  See \cite[Section 4]{Ruan04} for detail.

In what follows, for ultraweakly closed subspaces $A\subset M$ and $B\subset N$, where $M$ and $N$ are von Neumann algebras, $A\overline\otimes B$ denotes the ultraweak closure of the algebraic tensor product $A\otimes B$.

We shall use the following  results from \cite{Ruan04}, which we summarize as a theorem.

\begin{theorem}[Ruan \cite{Ruan04}]\label{thm:0316161}
Let $V$ be a W$^*$-TRO acting on separable Hilbert spaces.

{\rm (i)} \cite[Theorem 3.2]{Ruan04}
If $V$ is a stable W*-TRO, then $V$ is TRO-isomorphic to $M(V)$ and to $N(V)$.\smallskip

{\rm (ii)} \cite[Corollary 4.3]{Ruan04}
If $V$ is a  W*-TRO of one of the types $I_{\infty,\infty}, II_{\infty,\infty}$ or $III$, then $V$ is a stable W*-TRO, and hence TRO-isomorphic to a von Neumann algebra.
\smallskip

{\rm (iii)} \cite[Theorem 4.4]{Ruan04}
If $V$ is a  W*-TRO of  type $II_{1,\infty}$ (respectively $II_{\infty,1}$), then $V$ is  TRO-isomorphic to $B(H,\IC)\overline\otimes M$ (respectively $B(\IC,H)\overline\otimes N$), where $M$ (respectively $N$) is a von Neumann algebra of type $II_1$.
\end{theorem} 

Because taking a transpose is a triple isomorphism, we have the following consequence of Theorem~\ref{thm:0316161}(iii).
\begin{lemma}\label{lem:0316161}
A W*-TRO of type $II_{1,\infty}$ is triple isomorphic to a W*-TRO of type $II_{\infty,1}$. More precisely, 
$B(H,\IC)\overline\otimes M$ is triple isomorphic to $B(\IC,H)\overline\otimes M^t$, where $x^t=Jx^*J$, for $x\in M\subset B(H)$ and $J$ is a conjugation on $H$.
\end{lemma}

\begin{proposition}\label{thm:0316162}
Let $V$ be a W*-TRO.
\begin{description}
\item[(i)]
If $V$ acts on a separable Hilbert space and  is  of one of the types $I_{\infty,\infty}, II_{\infty,\infty}$ or $III$,  then every triple derivation of $V$ is an inner triple derivation and every TRO derivation of $V$ is an inner TRO-derivation.
\item[(ii)] If every TRO-derivation of any W$^*$-TRO of type $II_{1,\infty}$ has only inner TRO-derivations, then every TRO-derivation of any W$^*$-TRO of type $II_{\infty,1}$ has only inner TRO-derivations. The converse also holds.
\end{description}
\end{proposition}
\begin{proof}
(i) is an immediate consequence of Theorem~\ref{thm:0316161}(ii), Theorem~\ref{thm:0228161} and Proposition~\ref{prop:0307161}(iv).
(ii) is an immediate consequence of  Lemma~\ref{lem:0316161} and Proposition~\ref{prop:0307161}(iv).
\end{proof}

It follows from Remark~\ref{rem:3.5}  that if $M$ is  a von Neumann algebra of type $II_\infty$ or III and
$H$ is a separable Hilbert space, then $B(\IC,H)\overline\otimes M$ is triple isomorphic to a von Neumann algebra and hence has only inner TRO-derivations, giving alternate proofs of parts of Proposition~\ref{thm:0316162}(i).\smallskip

  By \cite[Theorem 4.1]{Ruan04}, if $V$ is a  W*-TRO of  type I, then $V$ is TRO-isomorphic to $\oplus_\alpha L^\infty(\Omega_\alpha)\overline\otimes B(K_\alpha,H_\alpha)$.  In the next two results, we consider the related TRO $C(\Omega,B(H,K))$, where $\Omega$ is a compact Hausdorff space.

\begin{lemma}\label{lem:0726171}
Let $E$ be a TRO and $\Omega$ a compact Hausdorff space.  
\begin{description}
\item[(i)] If every TRO derivation of $V:=C(\Omega,E)$ is an inner TRO derivation, then the same holds for $E$.
\item[(ii)] If every triple derivation of $V:=C(\Omega,E)$ is an inner triple derivation, then the same holds for $E$.
\end{description}
\end{lemma}
\begin{proof}
By Proposition~\ref{prop:0307161}(iv), it is sufficient to prove (i).

If $D$ is a  TRO derivation of $E$, then $\delta f(\omega):=D(f(\omega))$ is a  TRO derivation of $V$, as is easily checked. Suppose every TRO derivation of $V$ is an inner TRO derivation.  Then $\delta f=\alpha f+f\beta$, where $\alpha=-\alpha^*=\sum_i x_iy_i^*$ for some $ x_i,y_i\in V$, and $\beta=-\beta^*=\sum_i z_j^*w_j$ for some $z_j,w_j\in V$.

For $a\in E$, let $1\otimes a\in V$ be the constant function equal to $a$. Then $D(a)=D((1\otimes a)(\omega))=\delta(1\otimes a)(\omega)=\alpha(\omega)a+a\beta(\omega)$ for all $\omega\in\Omega$.
Since $\alpha(\omega)^*=-\alpha(\omega)\in EE^*$ and $\beta(\omega)^*=-\beta(\omega)\in E^*E$, $D$ is an inner TRO derivation of $E$.
\end{proof}

Recall from Theorem~\ref{prop:0307162}(iii) that the TRO $B(H,K)$  supports outer TRO derivations if and only if it is infinite dimensional and $\dim H\ne \dim K$.

\begin{proposition}\label{prop:0726171}
If $V=\oplus_\alpha C(\Omega_\alpha,E_\alpha)$, where $E_\alpha=B(K_\alpha,H_\alpha)$ and if every triple derivation of $V$ is an inner triple derivation, then for every $\alpha$, either
$\dim E_\alpha<\infty$ or $\dim K_\alpha=\dim H_\alpha$.
\end{proposition}
\begin{proof}
Let $\delta$ be a triple derivation of $V$, and let $\delta_\alpha=\delta|_{C(\Omega_\alpha,E_\alpha)}$, which is a triple derivation of the weak*-closed ideal $C(\Omega_\alpha,E_\alpha)$. Then $\delta(\{f_\alpha\})=\{\delta_\alpha f_\alpha\}$.  Moreover if $\delta$ is an inner triple derivation, say $\delta=\sum_i\delta(a^i,b^i)$ for $a^i=\{a_\alpha^i\},b^i=\{b_\alpha^i\}\in  V$, then 
$\delta_\alpha=\sum_i\delta(a^i_\alpha,b^i_\alpha)$ is an inner triple derivation of $C(\Omega_\alpha,E_\alpha)$.

Now suppose that every  triple derivation of $V$ is an inner triple derivation, and that for some $\alpha_0$, $E_{\alpha_0}$ is infinite dimensional and $\dim K_{\alpha_0}$ does not equal $\dim H_{\alpha_0}$.  Then, as noted above,  there is an outer triple derivation $D$ of $C(\Omega_{\alpha_0},E_{\alpha_0})$.   Then the triple derivation on $V$ which is zero on $C(\Omega_\alpha,E_\alpha)$ for $\alpha\ne \alpha_0$ and equal to $D$ on 
$C(\Omega_{\alpha_0},E_{\alpha_0})$, cannot be inner by the preceding paragraph, which is a contradiction.
\end{proof}

\section{Some questions left open}

\begin{questions}
It remains to complete the results of Theorem~\ref{prop:0307162} to include the cases where $p$ is a finite projection in a continuous von Neumann algebra, or when $p$ is arbitrary and $M$ is a general von Neumann algebra of type I.
As a possible tool for the first question, we note that there is an alternate proof of    Proposition~\ref{prop:0307161} (ii), in the case $X=pM$, $p$ finite, using the technique in  \cite[Section II.B]{HKPR16}.
\end{questions}

\begin{questions}
Besides the problem of extending the known cases to non separable Hilbert spaces, the cases left open  in Proposition~\ref{thm:0316162} for arbitrary W*-TROs are those of types $II_{1,1}$ and  $II_{1,\infty}$ (the latter being equivalent to $II_{\infty,1}$).
\end{questions}

\begin{questions}\label{questions 3}
Let $E$ be a W*-TRO, and let  $V=\oplus_\alpha L^\infty(\Omega_\alpha)\overline\otimes B(K_\alpha,H_\alpha)$ be a W*-TRO of type~I.

\begin{itemize}
\item If  every derivation of the W*-TRO $L^\infty(\Omega)\overline\otimes E$ is inner, does it follow that every derivation of $E$ is inner? 
\item If every derivation of 
$V$ is inner, does it follow that  $\dim B(K_\alpha,H_\alpha)<\infty$, for all $\alpha$; or $\sup_\alpha \dim B(K_\alpha,H_\alpha)<\infty$?
\item If $\sup_\alpha \dim B(K_\alpha,H_\alpha)<\infty$, does it follow that every derivation of 
$V$ is inner?
\end{itemize}
\end{questions}

\begin{remark}With respect to Questions 3,\smallskip

(i) In the first bullet, if $E$ had a separable predual, then a variant of \cite[1.22.13]{sakai} would state that $L^\infty(\Omega)\overline\otimes E=L^\infty(\Omega,E)$ and the technique in Lemma~\ref{lem:0726171} could be used.\smallskip

(ii)
In the first bullet, suppose that $E=pM$, with $M$ a von Neumann algebra in $B(H)$ and $p$ a projection in $M$, and let $D$ is a derivation of $E$.  Then $\delta:=id\otimes D$ is a derivation of $V=L^\infty(\Omega)\overline\otimes E$.  Assuming that $\delta$ is inner, there exist $\alpha=-\alpha^*\in VV^*=L^\infty(\Omega)\overline\otimes (\overline{EE^*})$ ($\overline{EE^*}$ denoting the weak closure) and $\beta\in V^*V=L^\infty(\Omega)\overline\otimes (\overline{E^*E})$, such that 
\[
1\otimes Dx=\alpha(1\otimes x)+(1\otimes x)\beta,\quad (x\in E).
\] 
We have $EE^*=pMp\subset B(pH)$, $\overline{E^*E}\subset B(H)$, and $L^\infty(\Omega)\subset B(L^2(\Omega))$.  For each $\varphi\in B(L^2(\Omega))_*$, let $R_\varphi:B(L^2(\Omega)\otimes pH)\rightarrow B(pH)$ be the slice map of Tomiyama defined by $R_\varphi(f\otimes x)=\varphi(f)x$ (\cite[Lemma 7.2.2]{EffRua00}).

Since $VV^*$ is the ultraweak closure of $L^\infty(\Omega)\otimes EE^*$, by the weak*-continuity of $R_\varphi$, we have 
\[
\varphi(1)Dx=R_\varphi(\alpha)x+xR_{\varphi}(\beta)
\]
with $R_\varphi (\alpha)\in EE^*$ and  $R_{\varphi}(\beta)\in \overline{E^*E}$. Thus, if $\dim H=\dim pH$, or if $E$ is finite dimensional, then  $R_{\varphi}(\beta)\in E^*E$, so that $D$ is an inner TRO-derivation, (take $\varphi$ to be a normal state so that $R_\varphi$ is self-adjoint and $R_\varphi(\alpha)^*=-R_\varphi(\alpha)$ and $R_\varphi(\beta)^*=-R_\varphi(\beta)$.)
In general, $D$ could be called a 
``quasi-inner'' TRO-derivation.\smallskip

(iii) In the second bullet, if each $B(K_\alpha,H_\alpha)$ had a separable predual, then a variant of \cite[1.22.13]{sakai} would state that $L^\infty(\Omega_\alpha)\overline\otimes B(K_\alpha,H_\alpha)=L^\infty(\Omega_\alpha,B(K_\alpha,H_\alpha))$ and the technique in Proposition~\ref{prop:0726171} could be used.
\end{remark}

\end{document}